\documentclass[12pt,a4paper]{amsart}
\usepackage{amssymb,color}

\usepackage[english]{babel}          
\usepackage{verbatim}
\usepackage[T1]{fontenc}

\usepackage{epstopdf}

\usepackage{floatflt,graphicx}
\usepackage{color,xcolor}
\usepackage{enumerate}
\usepackage{animate}
\usepackage{amsmath}
\usepackage{amsthm}

\newcounter{minutes}
\divide\time by 60
\newcounter{hours}
\multiply\time by 60 \addtocounter{minutes}{-\time}
\title{comparison theorems for hyperbolic type metrics}

\author{Oleksiy Dovgoshey}
\address{
National Academy of Sciences of Ukraine, Institute of Mathema\-tics,
Tereschenkivska str. 3, Kiev - 4, 01601, Ukraine
}
\curraddr{}
\email{aleksdov@mail.ru}

\author{Parisa Hariri}
\address{Department of Mathematics and Statistics,
  University of Turku, Turku, Finland}
\curraddr{}
\email{parisa.hariri@utu.fi}

\author{Matti Vuorinen}
\address{Department of Mathematics and Statistics,
  University of Turku, Turku, Finland}
\email{vuorinen@utu.fi}



\date{}

\dedicatory{}

\commby{}

\theoremstyle{plain}
\newtheorem{thm}[equation]{Theorem}
\newtheorem{cor}[equation]{Corollary}
\newtheorem{lem}[equation]{Lemma}

\newtheorem{prop}[equation]{Proposition}

\theoremstyle{definition}
\newtheorem{defn}[equation]{Definition}

\theoremstyle{remark}
\newtheorem{rem}[equation]{Remark}

\newtheorem{nonsec}[equation]{}

\numberwithin{equation}{section}

\DeclareMathOperator{\dist}{dist}

\newcommand{\beq}{\begin{equation}}
\newcommand{\eeq}{\end{equation}}

\newcommand{\bequu}{\begin{eqnarray*}}
\newcommand{\eequu}{\end{eqnarray*}}

\newcommand{\bequ}{\begin{eqnarray}}
\newcommand{\eequ}{\end{eqnarray}}

\newcommand{\R}{\mathbb{R}^2}

\newcommand{\Hn}{ {\mathbb{H}^n} }
\newcommand{\Bn}{ {\mathbb{B}^n} }
\newcommand{\Rn}{ {\mathbb{R}^n} }

\newcommand{\arch}{\,\textnormal{arch}}
\newcommand{\sh}{\,\textnormal{sh}}
\newcommand{\ch}{\,\textnormal{ch}}
\renewcommand{\th}{\,\textnormal{th}}

\begin{document}
\def\thefootnote{}
\footnotetext{ \texttt{File:~\jobname .tex,
           printed: \number\year-\number\month-\number\day,
           \thehours.\ifnum\theminutes<10{0}\fi\theminutes}
} \makeatletter\def\thefootnote{\@arabic\c@footnote}\makeatother

\begin{abstract}
The connection between several hyperbolic type metrics is studied in subdomains of the Euclidean space.
In particular, a new metric is introduced and compared to the distance ratio metric.
\end{abstract}

\maketitle



\section{Introduction}\label{section1}

\setcounter{equation}{0}

The notion of a metric, introduced by M. Fr\'echet in his thesis in 1906,
has become a basic tool in many areas of mathematics  \cite{dd}.  In geometric function theory, numerous metrics are extensively used in addition
to the Euclidean and  hyperbolic distances. Recently also metrics of hyperbolic type have become standard
tools in various areas of metric geometry. One of these metrics is the quasihyperbolic
metric of a domain $G\subsetneq \Rn\,.$ Although this metric has numerous applications, many of its basic properties
are still open. For instance, questions about the convexity of balls and properties of geodesics have been studied
by several people very recently \cite{mv, rt, krt}.
Given two points  $x,y\in G$ it is usually impossible to compute the quasihyperbolic distance between
them. Therefore various estimates in terms of quantities that are simple to compute are needed.
Some of the metrics that provide either upper or lower bounds for the quasihyperbolic
metric are the distance ratio metric, the triangular ratio metric and
the visual angle metric  \cite{chkv, himps, klvw, hvw}.

In this paper we study a generalization of a metric introduced by P. H\"ast\"o \cite{h}
in the one-dimensional case.  We also find estimates for it in terms of the aforementioned metrics.
Our main result is the following theorem.


\begin{thm}\label{th4}
Let $D$ be a nonempty open set in a metric space $(X, \rho)$ and let $\partial D \neq\varnothing$. Then the function
$$
h_{D,c}(x,y) = \log\left(1+c\frac{\rho(x,y)}{\sqrt{d_D(x)d_D(y)}}\right),
$$ where $d_{D}(x)=\dist (x, \partial D),$
is a metric for every $c \geqslant 2$. The constant $2$ is best possible here.
\end{thm}

H\"ast\"o's work \cite[Lemma 5.1]{h} also covers the case when $X=\Rn$ and $D=\Rn\setminus\{0\}$ in which case the best constant $c=1$.

The metric of Theorem \ref{th4} has its roots in the study of so called distance ratio metric or $j$-metric, see \eqref{jmetric} below, which is recurrent in the study of quasiconformal maps.
See also \cite[2.43]{vu} for a related quantity, which we will study below in Lemma \ref{set5}.

Further motivation for the study of the $h_{G,c}$-metric comes from applications to quasiconformal homeomorphisms $f:G \to f(G)$ between domains $G, f(G) \subsetneq {\mathbb R}^n\,.$ Our main applications are given in Section \ref{section4} where we show, for instance, that $h_{G,c}$-bilipschitz homeomorphisms are quasiconformal. We also
find two-sided bounds for the  $h_{G,c}$-metric in terms of the distance ratio metric. Some of the results of this paper were recently applied in \cite{na}.

\section{Preliminary results }\label{section 2}


\begin{nonsec}{\bf Hyperbolic metric.}
The hyperbolic metric $\rho_{\mathbb{H}^n}$ and $\rho_{\mathbb{B}^n}$ of the upper
half space ${\mathbb{H}^n} = \{ (x_1,\ldots,x_n)\in {\mathbb{R}^n}:  x_n>0 \} $
and of the unit ball ${\mathbb{B}^n}= \{ z\in {\mathbb{R}^n}: |z|<1 \} $ can be defined as follows.
By \cite[p.35]{b} we have for  $ x,y\in \mathbb{H}^n$
\begin{equation}\label{cro}
\ch{\rho_{\mathbb{H}^n}(x,y)}=1+\frac{|x-y|^2}{2x_ny_n},
\end{equation}
and by \cite[p.40]{b} for $x,y\in\mathbb{B}^n$
\begin{equation}\label{sro}
\sh{\frac{\rho_{\mathbb{B}^n}(x,y)}{2}}=\frac{|x-y|}{\sqrt{1-|x|^2}{\sqrt{1-|y|^2}}}\,.
\end{equation}
From \eqref{sro} we easily obtain
\begin{equation}\label{tro}
\th{\frac{\rho_{\mathbb{B}^n}(x,y)}{2}}=\frac{|x-y|}{\sqrt{|x-y|^2+(1-|x|^2)(1-|y|^2)}}\,. 
\end{equation}
\end{nonsec}
It is a well-know basic fact that
for both $\mathbb{B}^n$ and $\mathbb{H}^n$ one can define the hyperbolic metric using absolute ratios, see \cite{a}, \cite{b}, \cite[(2.21)]{vu}. Because of the M\"obius invariance of the absolute ratio we may define for every M\"obius transformation $g$ the hyperbolic metric in $g(\mathbb{B}^n)$. This metric will be denoted by $\rho_{g(\mathbb{B}^n)}$. In particular, if $g: {\mathbb{B}^n}\to {\mathbb{H}^n}$ is a M\"obius transformation with $g({\mathbb{B}^n})={\mathbb{H}^n},$ then for all $x,y \in
{\mathbb{B}^n}$ there holds $\rho_{\mathbb{B}^n}(x,y)= \rho_{\mathbb{H}^n}(g(x),g(y))\,.$

\begin{nonsec}{\bf Distance ratio metric.}\label{jmetric}
For a proper nonempty open subset $D \subset {\mathbb R}^n\,$ and for all
$x,y\in D$, the  distance ratio
metric $j_D$ is defined as
$$
 j_D(x,y)=\log \left( 1+\frac{|x-y|}{\min \{d(x,\partial D),d(y, \partial D) \} } \right)\,.
$$
The distance ratio metric was introduced by F.W. Gehring and B.P. Palka
\cite{gp} and in the above simplified form in \cite{vu0}. If there is no danger of confusion, we write $d(x)= d(x,\partial D)=\dist (x, \partial D)\,.$
\end{nonsec}

\begin{prop}\label{4}
If $c>0$ we have
\begin{enumerate}
\item
\[
\sqrt{2(\ch\rho_{\Hn}(x,y)-1)}=\frac{e^{h_{\Hn,c}(x,y)}-1}{c}\,\,, \quad  \text{for all} \, \, x,y \in \Hn \, ,
\]
\item
\[
\sh{\frac{\rho_{\mathbb{B}^n}(x,y)}{2}}\leq \frac{e^{h_{\Bn,c}(x,y)}-1}{c}\leq 2 \sh{\frac{\rho_{\mathbb{B}^n}(x,y)}{2}}\,\,,\quad  \text{for all} \, \, x,y \in \Bn \, .
\]
\end{enumerate}

\end{prop}
\begin{proof}
(1) By \eqref{cro}, for $x, y\in\Hn$ we see that
\[
\sqrt{2(\ch\rho_{\Hn}(x,y)-1)}=\frac{|x-y|}{\sqrt{x_n y_n}}=\frac{e^{h_{\Hn,c}(x,y)}-1}{c}.
\]

(2) By \eqref{sro}  for $x, y\in\Bn$
\bequu
\sh{\frac{\rho_{\mathbb{B}^n}(x,y)}{2}} &=& \frac{|x-y|}{\sqrt{1-|x|^2}{\sqrt{1-|y|^2}}}\\
&\leq & \frac{|x-y|}{\sqrt{1-|x|}{\sqrt{1-|y|}}}\\
& \leq & \frac{2|x-y|}{\sqrt{1-|x|^2}{\sqrt{1-|y|^2}}} \, ,
\eequu
so that the function  $h_{\Bn,c}(x,y)$ satisfies
\[
\sh{\frac{\rho_{\mathbb{B}^n}(x,y)}{2}}\leq \frac{e^{h_{\Bn,c}(x,y)}-1}{c}\leq 2 \sh{\frac{\rho_{\mathbb{B}^n}(x,y)}{2}} \, .\qedhere
\]
\end{proof}


\begin{lem}
Let $g:\Bn\to \Bn$ be a M\"obius transformation with  $g(\Bn) = \Bn.$ Then for $c>0$ the inequality
\[
h_{\Bn,c}(g(x),g(y))\leq 2 h_{\Bn,c}(x,y)
\]
holds for all $x, y\in\Bn\, . $
\end{lem}
\begin{proof}

If $g:\Bn\to \Bn$ is a M\"obius transformation, then $$\rho_{\Bn}(g(x), g(y) )=\rho_{\Bn}(x,y)$$
and  by Proposition \ref{4} (2)
\[
\sh{\frac{\rho_{\mathbb{B}^n}(x,y)}{2}}\leq \frac{e^{h_{\Bn,c}(g(x), g(y))}-1}{c}\leq 2 \sh{\frac{\rho_{\mathbb{B}^n}(x,y)}{2}} \, .
\]
Together with the Bernoulli inequality \cite[(3.6)]{vu1}, this yields
\[
h_{\Bn,c}(g(x), g(y))\leq \log\left(1+2c \sh{\frac{\rho_{\mathbb{B}^n}(x,y)}{2}} \right)\leq 2 h_{\Bn,c}(x,y)\, .\qedhere
\]
\end{proof}



\begin{prop}\label{2.8}
For $c, t>0$, let $f(t)=\log\left(1+2c \sh{\frac{t}{2}} \right).$ Then the double inequality $$\frac{c}{2(1+c)}t< f(t) < ct$$ holds for $c\geq \frac{1}{2}$ and $t>0\, . $
\end{prop}

\begin{proof}
We first show that
\beq\label{2.9}
f'(t)< c, \quad \text{for all}\, c\geq \frac{1}{2},\, t>0 \, .
\eeq
This inequality is equivalent to
\[
\frac{\sqrt{1+\sh^2{\frac{t}{2}}}}{1+2c \sh{\frac{t}{2}}}<1
\]
which clearly holds for $c\geq \frac{1}{2}\, . $

We next show that
\beq\label{2.10}
f'(t)> \frac{c}{2(1+c)}, \quad \text{for all}\, c>0,\, t>0 \, .
\eeq
From the definition of $\sh$ and $\ch$ it readily follows that for all $t\geq 0$, $\ch{\frac{t}{2}}> \frac{1}{2} e^{\frac{t}{2}}$ and $\sh{\frac{t}{2}}< \frac{1}{2} e^{\frac{t}{2}}\, . $
Therefore
\[
f'(t)>\frac{c \frac{1}{2} e^{\frac{t}{2}}}{1+c e^{\frac{t}{2}}}=\frac{c}{2}\cdot\frac{u}{1+cu};\, u= e^{\frac{t}{2}} \, .
\]
Writing $g(u)=\frac{u}{1+cu}$ we see that $g$ is increasing with
\[
\lim_{u\to 1} g(u)=\frac{1}{1+c}\quad\mbox{and}\quad \lim_{u\to \infty} g(u)=\frac{1}{c} \, .
\]
Therefore we have
\[
f'(t)> \frac{c}{2} \min\left\{\frac{1}{c},\frac{1}{c+1}\right\}= \frac{c}{2(1+c)} \, .
\]
Because $f(0)=0$, \eqref{2.9} and \eqref{2.10} imply the desired conclusion.
\end{proof}


\begin{lem}
If $g:\Bn \to \Hn$ is a M\"obius transformation and $x, y\in \Bn$ then for $c>0$
\[
h_{\Hn,c}(g(x), g(y))\leq 2 h_{\Bn,c}(x,y) \, .
\]
\end{lem}
\begin{proof}
By Proposition \ref{4} (1) and \cite[(2.21)]{vu}
\bequu
\frac{e^{h(g(x), g(y))}-1}{c} & = & \sqrt{2(\ch\rho_{\Hn}(g(x) , g(y))-1)}\\
&=& \sqrt{2(\ch\rho_{\Bn}(x,y)-1)}=e^{\rho_{\Bn}(x,y)/2}-e^{-\rho_{\Bn}(x,y)/2}\\
&=& 2\sh(\rho_{\Bn}(x,y)/2)\, .
\eequu
Together with  Bernoulli's inequality, this yields
\[
h_{\Hn,c}(g(x), g(y))= \log\left(1+2c\sh{\frac{\rho_{\Bn}(x,y)}{2}}\right)\leq 2 h_{\Bn,c}(x,y) \, .\qedhere
\]

\end{proof}

\begin{lem}\label{set5}
For a nonempty open set  $D\subsetneq \Rn$ and $x, y\in D$, let
$$
  \varphi_D(x,y)=\log \left(1+\max\left\{{|x-y|\over \sqrt{d(x)d(y)}}\;,\;
           {|x-y|^2\over d(x)d(y)}\right\}\right)\,.
$$
Then the double inequality
\beq\label{2.10*}
j_D(x,y)/2 \leq\varphi_D(x,y)\leq 2j_D(x,y)
\eeq
holds.
\end{lem}

\begin{proof}
For the first inequality in \eqref{2.10*}, we may assume that $d(x)\leq d(y)\, . $ We claim that
\[
\frac{j_D(x,y)}{2}\leq \log\left(1+\frac{|x-y|}{\sqrt{d(x)d(y)}}\right).
\]
This inequality is equivalent to
\[
\log\left(1+\frac{|x-y|}{d(x)}\right)\leq 2\log\left(1+\frac{|x-y|}{\sqrt{d(x)d(y)}}\right)\, ,
\]
and to
\[
\frac{1}{d(x)}\leq \frac{2}{\sqrt{d(x)d(y)}}+\frac{|x-y|}{d(x)d(y)} \, .
\]
This last inequality holds by the triangle inequality, because $d(x)\leq d(y)\, . $

For the second inequality
\begin{eqnarray*}
\varphi_D(x,y) &\leq & \log\left(1+\frac{|x-y|}{\sqrt{d(x)d(y)}}+\frac{|x-y|^2}{d(x)d(y)}\right)\\
&\leq & 2\log\left(1+\frac{|x-y|}{\sqrt{d(x)d(y)}}\right)\leq 2j_D(x,y)\, .
\end{eqnarray*}

\end{proof}
The above proof also yields the following result.


\begin{cor}
For a nonempty open set  $D\subsetneq \Rn$ and $x, y\in D$,
\[
j_D(x,y)/2 \leq h_{D,1}(x,y)\leq \varphi_D(x,y)\leq 2h_{D,1}(x,y)\leq 2j_D(x,y) \,,
\]
holds.
\end{cor}

\begin{rem}
The function $\varphi_{\mathbb{B}^2}(x,y)$ is not a metric, because the inequality $$\varphi_{\mathbb{B}^2}(t,0)+\varphi_{\mathbb{B}^2}(0,-t)\ge\varphi_{\mathbb{B}^2}(t,-t)$$ fails for $t\in (0,1)\, . $
\end{rem}


\section{Proof for the main result}\label{section3}

Throughout this section we assume that $D$ is a nonempty open set in a metric space $(X, \rho)$ and $A$ is a nonempty subset of $X\setminus D.$ For every $x\in D$ write $d_{D, A}(x)=\dist (x, A).$


\begin{lem}\label{l1}
The inequality
\begin{equation}\label{eq1}
d_{D, A}(x) \leqslant \rho(x,y) + d_{D, A}(y)
\end{equation}
holds for all $x$, $y \in D\, . $
\end{lem}
\begin{proof}
Inequality \eqref{eq1} follows directly from the triangle inequality.
\end{proof}

Let
$$
h_{D}^{A}(x,y) = \log\left(1+2\frac{\rho(x,y)}{\sqrt{d_{D, A}(x)d_{D, A}(y)}}\right)
$$
for all $x$, $y \in D$ and let $h_{D}^{\partial D}=h_D$ for short. Thus, $h_D=h_{D, 2}.$


\begin{prop}\label{p2}
The triangle inequality
\begin{equation}\label{eq2}
h_{D}^{A}(x,y) \leqslant h_{D}^{A}(x,z) + h_{D}^{A}(z,y)
\end{equation}
holds for all $x, y, z \in D$ and every $(X, \rho),$ and every nonempty $A\subseteq X\setminus D$ if the triangle inequality
\begin{equation}\label{eq3}
h_{I}(p,s) \leqslant h_{I}(p,q) + h_{I}(q,s)
\end{equation}
holds for every open interval $I = (a,b)$ and all $p, s, q \in I$ which satisfy the conditions $p<q<s$ and $d_I(p) = |a-p|$, $d_I(s)=|s-b|\, . $
\end{prop}

\begin{proof}
Inequality \eqref{eq2} is trivial if $x=y$ or $x=z$ or $z=y\, . $ Let $\rho(x,y) \neq 0$, $\rho(x,z) \neq 0$ and $\rho(z,y) \neq 0\, . $ Inequality \eqref{eq2} can be written in the form
\begin{equation}\label{eq4}
\rho(x,y) \leqslant \rho(x,z) \frac{\sqrt{d_{D, A}(y)}}{\sqrt{d_{D, A}(z)}} + \rho(z,y) \frac{\sqrt{d_{D, A}(x)}}{\sqrt{d_{D, A}(z)}} + 2 \frac{\rho(x,z)\rho(z,y)}{d_{D, A}(z)}.
\end{equation}
Write
$$
\Delta(x,y,z) = \min\{\rho(x,z)+d_{D, A}(x), \rho(y,z)+d_{D, A}(y)\}.
$$
It follows from Lemma \ref{l1} that
\begin{equation}\label{eq5}
d_{D, A}(z) \leqslant \Delta(x,y,z).
\end{equation}
If there are $x_1, y_1, z_1 \in D$ such that \eqref{eq4} does not hold for $x=x_1$, $y=y_1$, $z=z_1$, then using \eqref{eq5} we obtain
\begin{multline}\label{eq6}
\rho(x_1, y_1) > \rho(x_1,z_1) \frac{\sqrt{d_{D, A}(y_1)}}{\sqrt{\Delta(x_1,y_1,z_1)}} \\
{} + \rho(z_1,y_1) \frac{\sqrt{d_{D, A}(x_1)}}{\sqrt{\Delta(x_1,y_1,z_1)}} + 2 \frac{\rho(x_1,z_1)\rho(z_1,y_1)}{\Delta(x_1,y_1,z_1)} \,.
\end{multline}
Write
\begin{equation}\label{eq7}
k = \frac{\rho(x_1, y_1)}{\rho(x_1,z_1) + \rho(z_1,y_1)}\,.
\end{equation}
Let us consider an open interval $I = (p', s') \subset \mathbb R$ and the points $p, q, s \in (p', s')$ such that
$$
p' < p< q<s <s'\,,
$$
and
\begin{equation}\label{eq7'}
\begin{split}
|p-q|=k\rho(x_1, z_1), \quad |q-s| = k\rho(z_1,y_1)\,, \\
|p'-p|=k d_D(x_1), \quad  |s'-s|=k d_D(y_1)\,,
\end{split}
\end{equation}
(see Figure~\ref{fig1}).

\begin{figure}[h]
\begin{picture}(12,5)
\put(0,2.5){\line(1,0){12}}
\put(-0.05,2.38){|}
\put(-0.1,1.8){$p'$}
\put(11.95,2.38){|}
\put(11.9,1.8){$s'$}
\put(1.95,2.38){|}
\put(1.9,1.8){$p$}
\put(4.95,2.38){|}
\put(4.9,1.8){$q$}
\put(9.95,2.38){|}
\put(9.9,1.8){$s$}
\put(6,1){\vector(1,0){4}}
\put(6,1){\vector(-1,0){4}}
\put(5.5,0.4){$\rho(x_1, y_1)$}
\put(1,3.5){\vector(1,0){1}}
\put(1,3.5){\vector(-1,0){1}}
\put(0.5,4){$d_I(p)$}
\put(11,3.5){\vector(1,0){1}}
\put(11,3.5){\vector(-1,0){1}}
\put(10.5,4){$d_I(s)$}
\end{picture}
\caption{}\label{fig1}
\end{figure}
From \eqref{eq7} it follows that
\begin{equation}\label{eq7*}
|p-s|=\rho(x_1, y_1)\,.
\end{equation}
Moreover, we claim that
\begin{equation}\label{eq8}
d_I(p) = |p-p'| \quad \text{and} \quad d_I(s) = |s-s'|\,.
\end{equation}
Indeed, the equality
$$
d_I(p) = \min\{|p-p'|, |p-s'|\}
$$
holds because
$$
d_{I}(p) = \dist (p,\{p',s'\})\,.
$$
Hence
\begin{equation}\label{eq9}
d_I(p) = |p-p'|
\end{equation}
holds if and only if
\begin{equation}\label{eq10}
|p-p'| \leqslant |p-s'|\,.
\end{equation}
Using \eqref{eq7'}, Lemma~\ref{l1} and the inequality $0<k \leqslant 1$ we obtain
\begin{multline*}
|p-s'| = |p-s|+|s-s'| = \rho(x_1,y_1) + k d_D(y_1) \\
\geqslant k(\rho(x_1,y_1) + d_D(y_1)) \geqslant kd_D(x_1) = |p-p'|\,.
\end{multline*}
Inequality \eqref{eq10} follows, so that equality \eqref{eq9} is proved. The equality
$$
d_I(s) = |s-s'|
$$
can be proved similarly.

By definition we have
$$
d_I(q) = \min\{|p'-q|, |s'-q|\}\,.
$$
Hence
\begin{multline}\label{eq11}
d_I(q) = k\Delta(x_1, y_1, z_1) \\
=k\min\{d_{D, A}(x_1) + \rho(x_1, z_1), d_{D, A}(y_1) + \rho(y_1, z_1)\}.
\end{multline}
Using \eqref{eq11}, \eqref{eq8}, \eqref{eq7'} and \eqref{eq7*} we can write \eqref{eq6} in the next form
\begin{multline}\label{eq12}
|p-s|> k^{-1}|p-q| \frac{\sqrt{d_I(s)}}{\sqrt{d_I(q)}} + k^{-1} |q-s| \frac{\sqrt{d_I(p)}}{\sqrt{d_I(q)}} \\
+ 2 k^{-1}\frac{|p-q| |q-s|}{d_I(q)}.
\end{multline}
Since $k^{-1}\geqslant 1$, inequality \eqref{eq12} implies
$$
|p-s|> |p-q| \frac{\sqrt{d_I(s)}}{\sqrt{d_I(q)}} + |q-s| \frac{\sqrt{d_I(p)}}{\sqrt{d_I(q)}}+ 2 \frac{|p-q| |q-s|}{d_I(q)}\,.
$$
The last inequality contradicts the triangle inequality for $h_I\, . $
\end{proof}


\begin{lem}\label{l3}
Let $I$ be the interval depicted in Figure~\ref{fig1}. Suppose that $d_I(p)=|p-p'|$, $d_I(s)=|s-s'|$ and $d_I(q)=|p-q|+|p-p'|\, . $ Then the inequality
\begin{equation}\label{eq14}
h_I(p,s)\leqslant h_I(p,q)+h_I(q,s)
\end{equation}
holds.
\end{lem}
\begin{proof}
Inequality \eqref{eq14} has the following equivalent form
\begin{equation}\label{eq15}
|p-s| \leqslant |p-q| \frac{\sqrt{d_I(s)}}{\sqrt{d_I(q)}} + |q-s| \frac{\sqrt{d_I(p)}}{\sqrt{d_I(q)}} + 2\frac{|p-q| |q-s|}{d_I(q)}\,.
\end{equation}
Suppose that $d_I(q)=|p-q|+|p-p'| < |q-s|+|s-s'|$ and there is $s_0\in (s, s')$ such that
\begin{equation}\label{eq13}
|p-q|+|p-p'| = |q-s|+|s-s_0|\,.
\end{equation}
Let us consider the interval $J=(p', s_0)\, . $

\begin{figure}[h]
\begin{picture}(12,2.5)
\put(0,1){\line(1,0){12}}
\put(-0.05,0.88){|}
\put(-0.1,0.3){$p'$}
\put(11.95,0.88){|}
\put(11.9,0.3){$s'$}
\put(1.95,0.88){|}
\put(1.9,0.3){$p$}
\put(4.95,0.88){|}
\put(4.9,0.3){$q$}
\put(9.45,0.88){|}
\put(9.4,0.3){$s$}
\put(11.45,0.88){|}
\put(11.4,0.3){$s_0$}
\put(1,1.5){\vector(1,0){1}}
\put(1,1.5){\vector(-1,0){1}}
\put(0.5,1.8){$d_I(p)$}
\put(10.5,1.5){\vector(1,0){1}}
\put(10.5,1.5){\vector(-1,0){1}}
\put(10,1.8){$d_I(s)$}
\end{picture}
\label{fig2}
\caption{}
\end{figure}
Then we have $d_J(s) < d_I(s)$ and $d_J(p)=d_I(p)$ and
$$
d_J(q)=d_I(q)=|p-q|+|p-p'|=|q-s|+|s-s_0|.
$$
Consequently, inequality \eqref{eq15} follows from the inequality
\begin{equation}\label{eq16}
|p-s| \leqslant |p-q| \frac{\sqrt{d_J(s)}}{\sqrt{d_J(q)}} + |q-s| \frac{\sqrt{d_J(p)}}{\sqrt{d_J(q)}}
+ 2\frac{|p-q| |q-s|}{d_J(q)}.
\end{equation}
Using a suitable shift and a scaling we may suppose that $q=0$ and $p'=-1$ and $s_0=1\, . $ Consequently \eqref{eq16} obtains the form
\begin{equation}\label{eq17}
|p|+|s| \leqslant |p| \sqrt{1-|s|} + |s|\sqrt{1-|p|} + 2 |s||p|,
\end{equation}
because $\sqrt{d_J(s)} = 1-|s|$, $\sqrt{d_J(p)} = 1-|p|$ and $\sqrt{d_J(0)}=1$ (see Figure 3).

\begin{figure}[h]
\begin{center}
\noindent
\begin{picture}(12,2.5)
\put(0,1){\line(1,0){12}}
\put(-0.05,0.88){|}
\put(-0.1,0.3){$-1$}
\put(11.95,0.88){|}
\put(11.9,0.3){$1$}
\put(1.95,0.88){|}
\put(1.9,0.3){$p$}
\put(5.95,0.88){|}
\put(5.9,0.3){$0$}
\put(9.45,0.88){|}
\put(9.4,0.3){$s$}
\put(1,1.5){\vector(1,0){5}}
\put(1,1.5){\vector(-1,0){1}}
\put(3,1.8){$d_J(p)$}
\put(7,1.5){\vector(1,0){5}}
\put(7,1.5){\vector(-1,0){1}}
\put(8.5,1.8){$d_J(s)$}
\end{picture}
\end{center}
\caption{}

\label{fig3}
\end{figure}

Inequality \eqref{eq17} holds if and only if
\begin{equation}\label{eq18}
\frac{1}{|p|} + \frac{1}{|s|} - \left(\frac{\sqrt{1-|s|}}{|s|} + \frac{\sqrt{1-|p|}}{|p|}\right) \leqslant 2.
\end{equation}
Let $f(x) = x - \sqrt{x^2 - x}$, $x \geqslant 1\, . $ Then
$$
f'(x) = 1 - \frac{1}{2}\left(x^2-x\right)^{-1/2}(2x-1),
$$
$$
f'(x)\left(x^2-x\right)^{1/2} = \left(x^2-x\right)^{1/2} - \left(x^2-x+\frac{1}{4}\right)^{1/2} .
$$
Hence $f'(x)<0$ for all $x > 1\, . $ It implies that $f$ is decreasing on $[1, \infty)\, . $

Putting $x=\frac{1}{|p|}$ and $x = \frac{1}{|s|}$ we obtain  from the inequality $f(x)\leqslant f(1)$ that

\begin{equation}\label{eq19}
\begin{split}
\frac{1}{|p|} - \frac{\sqrt{1-|p|}}{|p|} \leqslant \frac{1}{1} - \frac{\sqrt{1-1}}{1}=1\,,
\\
\frac{1}{|s|} - \frac{\sqrt{1-|s|}}{|s|} \leqslant \frac{1}{1} - \frac{\sqrt{1-1}}{1}=1\,.
\end{split}
\end{equation}
Inequality \eqref{eq18} follows.

Suppose now that the inequality
\begin{equation}\label{eq20}
|p-q|+|p-p'| < |q-s| + |s-s_0|
\end{equation}
holds for all $s_0 \in (s, s')\, . $ Letting $s' \to s$ we obtain from \eqref{eq15} that
\begin{equation}\label{eq21}
|p-s| \leqslant |q-s| \frac{\sqrt{d_I(p)}}{\sqrt{d_I(q)}} + 2\frac{|p-q| |q-s|}{d_I(q)}\,.
\end{equation}
It is sufficient to show that \eqref{eq21} holds. Using some shift and scaling we can put $q=1$ and $p'=0$ (see Figure 4).

\begin{figure}[h]
\begin{center}
\noindent
\begin{picture}(11,1.5)
\put(0,1){\line(1,0){9.5}}
\put(-0.05,0.88){|}
\put(-0.1,0.3){$0$}
\put(1.95,0.88){|}
\put(1.9,0.3){$p$}
\put(4.95,0.88){|}
\put(4.9,0.3){$1$}
\put(9.45,0.88){|}
\put(9.4,0.3){$s$}
\end{picture}
\end{center}
\caption{}
\end{figure}

Now inequality \eqref{eq21} has the form
\begin{equation}\label{eq22}
s-p \leqslant (s-1) \sqrt{p} + 2(1-p) (s-1)\,,
\end{equation}
because $d_I (q) = d_I (1) = |1-p| + |p-0| = 1\, . $ From \eqref{eq20} we obtain $s\geqslant 2\, . $ Write $p=x^2\, . $ Let us consider the function
\begin{multline*}
F(s,x) = x^2 - s + (s-1)x + 2(s-1) - 2x^2(s-1) \\
= (-2s+3)x^2 + (s-1)x + (s-2)\,.
\end{multline*}
If $s^* \in [2, \infty)$ is given, then the function $y = F(s^*, x)$ is a parabola and $(-2s^* + 3)< 0\, . $ Thus, $F(s^*, x)$ is concave. We have
\begin{equation}\label{eq23}
F(s^*, 0) = s-2 \geqslant 0 \text{ and } F(s^*, 1) = -2s^* + 3 + s^* -1 + s^* - 2 = 0\,.
\end{equation}
Hence \eqref{eq23} implies $F(s^*, x) \geqslant 0$ for all $x \in (0, 1)$ and every $s \geqslant 2\, . $ Inequality \eqref{eq22} follows.

\end{proof}


Theorem \ref{th4} follows from Lemma \ref{l3} and Proposition \ref{p2} with $A=\partial D.$

\begin{rem}
It is possible that
$h_{D,c}$ fails to be a metric if $0<c<2\, . $ To this end we consider $D=\mathbb{B}^2$, and fix $c\in (0,2)\, . $ Then for $0<r<1$ the triangle inequality $2h_{D,c}(0,r)\geq h_{D,c}(-r,r)$ is equivalent to
\beq\label{htriangle}
\frac{2}{\sqrt{1-r}}+\frac{cr}{1-r}\geq\frac{2}{1-r}.
\eeq
Now for large enough $r$, \eqref{htriangle} yields $c\geq 2$, which is contradiction.
\end{rem}


\begin{rem}
If a metric space $(X,\rho)$ is connected, then
\begin{equation*}\label{eq26}
\partial D\ne\varnothing
\end{equation*}
holds for every nonempty proper open subset $D$ of $X.$
Consequently, $h_{D, c}$ is a metric for all connected metric spaces $X$, proper open $D\subseteq X$ and $c\ge 2.$
\end{rem}

\section{Comparison results for $h_{G,c}$}\label{section4}

In this section we shall study the class of uniform domains which is recurrent in geometric function theory of Euclidean spaces \cite{gh,vu}. In the planar case so-called quasidisks, i.e., simply connected domains in the plane bounded by curve $\Gamma$, which is the image of the unit circle under a quasiconformal homeomorphism of $\R$ onto itself, form a well-known class of uniform domains.

We show that in uniform domains the $h_{G,c}$ and $j_G$ metrics are comparable. For this purpose we introduce the quasihyperbolic metric. Using these comparison results we then proceed to prove that $h_{G,c}$-metric is quasi-invariant under quasiconformal mappings.

For some basic facts about quasiconformal maps the reader is referred to \cite{v1}.

\begin{nonsec}{\bf Quasihyperbolic metric.}
Let $G$ be a proper subdomain of ${\mathbb R}^n\,\, . $ For all $x,\,y\in G$, the quasihyperbolic metric $k_G$ is defined as
$$k_G(x,y)=\inf_{\gamma}\int_{\gamma}\frac{1}{d(z,\partial G)}|dz|,$$
where the infimum is taken over all rectifiable arcs $\gamma$ joining $x$ and $y$ in $G$ \cite{gp}.

\end{nonsec}

It is a well-known basic fact \cite{gp}, that for all $x, y\in G$
\beq\label{qhj}
k_G(x,y)\geq j_G(x,y).
\eeq


\begin{defn}\label{uni}
A domain $G\subsetneq \Rn$ is said to be uniform, if there exists a constant $U\geq 1$ such that for all $x, y\in G$
\[
k_G(x,y)\leq U j_G(x,y).
\]
\end{defn}


\begin{lem}\label{4.4}
Let $G\subsetneq\Rn$ be a domain. Then for $c>0$ and all $x, y\in G$, we have
\begin{enumerate}
\item
\[
\frac{c}{2(1+c)}j_G(x,y)\leq \log\left(1+2c\sh{\frac{j_G(x,y)}{2}}\right)\leq h_{G,c}(x,y)\leq c j_G(x,y).
\]

If $x\in G$, $\lambda\in (0,1)$, $y\in \Bn(x,\lambda d(x))$ then
\item
\[
\frac{1-\lambda}{1+\lambda}j_G(x,y)\leq h_{G,c}(x,y).
\]
\end{enumerate}
\end{lem}

\begin{proof}
(1) Because $\sqrt{d(x)d(y)}\geq \min\{d(x), d(y)\}$, the Bernoulli inequality \cite[(3.6)]{vu} yields
\[
h_{G,c}(x,y)\leq \log\left(1+c\frac{|x-y|}{\min\{d(x), d(y)\}}\right)\leq c j_G(x,y).
\]
Next, by the triangle inequality we have $d(y)\leq d(x)+|x-y|$ and hence, for $d(x)\leq d(y),$
\bequu
(e^{h_{G,c}(x,y)}-1)/c &\geq& \frac{|x-y|}{\sqrt{d(x)(d(x)+|x-y|)}}\\
&=& \frac{t}{\sqrt{1+t}}=2\sh{\frac{j_G(x,y)}{2}}\quad \quad \mbox{where}\quad t=\frac{|x-y|}{d(x)}.
\eequu
If $d(y)< d(x)$ the argument is similar.

The lower bound follows from the Proposition \ref{2.8}.

(2) Observe that
\bequu
\sqrt{d(x)d(y)}&\leq & \max\{d(x), d(y)\} \leq (1+\lambda) d(x)\\
&= & \frac{1+\lambda}{1-\lambda}\cdot (1-\lambda) d(x)\leq  \frac{1+\lambda}{1-\lambda}\min\{d(x), d(y)\}
\eequu
and hence by Bernoulli's inequality
\[
h_{G,c}(x,y) \geq  \log\left(1+\frac{1-\lambda}{1+\lambda} \cdot \frac{|x-y|}{\min\{d(x), d(y)\}}\right)\geq \frac{1-\lambda}{1+\lambda} j_G(x,y).
\]
\end{proof}


\begin{cor}\label{4.5}
Let $G\subsetneq \Rn$ be a uniform domain. Then there exists a constant $d$ such that
\[
d k_G(x,y)\leq h_{G,c}(x,y)\leq c k_G(x,y)
\]
for all $x,y\in G$ and all $c>0 \, . $
\end{cor}

\begin{proof}
By \ref{uni} and \ref{4.4} (1) there exist constants $U>1$ and $c>1$ such that
\[
k_G(x,y)\leq U j_G(x,y)\leq U \left(\frac{2(1+c)}{c}\right) h_{G,c}(x,y).
\]
The second inequality follows from \ref{4.4} (1) and \eqref{qhj}.
\end{proof}


\begin{thm}\label{rhoh}
Let $G \in \{  {\mathbb B}^n , {\mathbb H}^n\}\,,$ and let
 $\rho_G$ stand for the respective hyperbolic metric. If $c \ge 2\,,$ then for all $x,y\in G$
 $$\frac{1}{c} h_{G,c}(x,y)\leq \rho_G(x,y) \le 2h_{G,c}(x,y).$$
\end{thm}
\begin{proof}
If $G=\Bn$, then by \eqref{tro}, \cite[(7.38),(7.53)]{avv},
\[
\rho_{\Bn}(x,y)=2\log{\frac{|x-y|+A[x,y]}{\sqrt{(1-|x|^2)(1-|y|^2)}}},
\]
\[
A[x,y]^2=|x-y|^2+(1-|x|^2)(1-|y|^2).
\]
The inequality $\rho_{\Bn}(x,y)\leq 2 h_{\Bn,c}(x,y)$ is equivalent to
\[
\quad \frac{|x-y|+A[x,y]}{\sqrt{(1-|x|^2)(1-|y|^2)}}\leq 1+c\frac{|x-y|}{\sqrt{(1-|x|)(1-|y|)}}.
\]
This inequality holds because
\[
A[x,y]\leq \sqrt{(1-|x|^2)(1-|y|^2)}+(c-1)|x-y|,
\]
which follows from the fact that $c\geq 2$, and $\sqrt{a^2+b^2}\leq a+b$, where $a=|x-y|$, and $b=\sqrt{(1-|x|^2)(1-|y|^2)}\, . $

If $G=\Hn$, then by \eqref{cro} and \cite[7.53]{avv},
\[
\rho_{\Hn}(x,y)=\log\left(1+\sqrt{u^2-1}\right),
\]

\[
u=1+\frac{|x-y|^2}{2x_n y_n}.
\]
By \cite[(2.14)]{vu}
\bequu
\rho_{\Hn}(x,y) &\leq & 2\log\left(1+\sqrt{2(u-1)}\right)\\
& = & 2\log\left(1+\frac{|x-y|}{\sqrt{x_n y_n}}\right)\\
& \leq & 2 \log\left(1+c \frac{|x-y|}{\sqrt{x_n y_n}}\right)\\
& = & 2h_{\Hn, c}(x,y).
\eequu
The lower bound follows from  \cite[Lemma 2.41(2)]{vu},  \cite[Lemma 7.56]{avv} (see also \cite[Lemma 3.2]{chkv}), and Lemma \ref{4.4}(1).
\end{proof}

F. W. Gehring and B. G. Osgood \cite{go} proved the following result, see also \cite[12.19]{vu1}.


\begin{thm}\label{12.19}
Let $f:G\to fG$ be a $K$-quasiconformal mapping where $G$ and $fG$ are proper subdomains of $\Rn\, . $ Then
\[
k_{fG}\left(f(x),f(y)\right)\leq c_1 \max\left\{k_G(x,y)^{\alpha},k_G(x,y)\right\}
\]
holds for all $x, y\in G$ where $\alpha=K_I(f)^{1/(1-n)}$ and $c_1$ depends only on $K_O(f)\, . $ Here $K_O$ and $K_I$ are the outer and inner dilatations.
\end{thm}
Applying Theorem \ref{12.19} we shall now prove the following two results.


\begin{lem}\label{4.7}
Let $f:G\to fG$ be a homeomorphism, where $G$ and $fG$ are proper subdomains of $\Rn\,,$ and suppose that for some $a\in (0,1)$
\[
j_{fG}\left(f(x),f(y)\right)\leq \frac{1}{a}\max\left\{j_G(x,y), j_G(x,y)^{a}\right\}
\]
for all $x, y\in G\, . $ Then for all $c>0$ and for all $x, y\in G$
\[
h_{fG,c}(f(x),f(y))\leq \frac{1}{A}\max \left\{h_{G,c}(x,y), h_{G,c}(x,y)^{a}\right\}
\]
where $A=A(c)\in (0,1)\,.$
\end{lem}

\begin{proof}
By Lemma \ref{4.4} (1) we see that
\bequu
h_{fG,c}\left(f(x),f(y)\right) &\leq & c j_{fG}\left(f(x),f(y)\right)\\
&\leq & \frac{c}{a} \max \left\{j_G(x,y), j_G(x,y)^{a}\right\}\\
&\leq & \frac{c}{a} \max \left\{\frac{2(1+c)}{c} h_{G,c}(x,y), \left(\frac{2(1+c)}{c}\right)^{a} h_{G,c}(x,y)^{a}\right\}\\
&= & \frac{1}{A} \max \left\{h_{G,c}(x,y), h_{G,c}(x,y)^{a}\right\}
\eequu
for all $x, y\in G$, where $A=\frac{a}{2(1+c)}\, . $
\end{proof}


\begin{thm}
Let $f:G\to fG\,$ be a $K$-quasiconformal map, where $G$ and $fG$ are proper subdomains of $\Rn\,,$  and let $G$ be a uniform domain. Then for all $c>0$ there exists $e\in (0,1)$ such that
\[
h_{fG,c}\left(f(x),f(y)\right)\leq \frac{1}{e} \max \left\{h_{G,c}(x,y), h_{G,c}(x,y)^{\alpha}\right\}
\]
where $\alpha=K_I(f)^{1/(1-n)}\, . $
\end{thm}

\begin{proof}
Because $G$ is uniform domain there exists constant $U>1$ such that
\[
k_G(x,y) \leq U j_G(x,y).
\]
By Corollary \ref{4.5} and Theorem \ref{12.19},
\bequu
h_{fG,c}(f(x),f(y))  &\leq & c k_{fG}(f(x),f(y)) \\
&\leq &  c c_1 \max\left\{k_G(x,y)^{\alpha},k_G(x,y)\right\}\\
&\leq & c c_1 \max\left\{(U\frac{2(1+c)}{c}h_{G,c}(x,y))^{\alpha},U\frac{2(1+c)}{c}h_{G,c}(x,y)\right\}\\
& = & \frac{1}{e} \max \left\{h_{G,c}(x,y), h_{G,c}(x,y)^{\alpha}\right\},
\eequu
for all $x, y\in G$, where $e=\frac{1}{2c_1(1+c)U}  \,.$
\end{proof}


\begin{thm}
Let $c>0$ and $f:G\to fG$ be a homeomorphism, where $G$ and $fG$ are proper subdomains of $\Rn\,,$ and suppose that there exists $L\geq 1$ such that for all $x, y\in G$
\[
h_{G,c}(x,y)/L\leq h_{fG,c}(f(x),f(y))\leq L h_{G,c}(x,y).
\]
Then $f$ is quasiconformal with linear dilatation $H(f)\leq L^2\, . $

\end{thm}
\begin{proof}
For $a, b\in G$ let
\beq\label{qc1}
m_G(a,b)=\min\{d(a), d(b)\},\quad U_G(a,b)=(e^{h_{G,c}(a,b)}-1)/c.
\eeq
Fix $z\in G$, $t\in (0,1/2)$ and $x, y\in G$ with
\beq\label{qc0}
|x-z|=|y-z|=td(z).
\eeq
It follows from the triangle inequality that the inequality  $d(f(w))\leq d(f(z))+|f(z)-f(w)|$ holds for $w\in G$ and hence
\[
U_{fG}(f(w),f(z))\cdot m_{fG}(f(w),f(z)) \leq  |f(w)-f(z)|
\]
\[
\leq  U_{fG}(f(w),f(z))\cdot m_{fG}(f(w),f(z))\sqrt{1+\frac{|f(w)-f(z)|}{m_{fG}(f(w),f(z))}}.
\]
It follows from \eqref{qc0} that
\beq\label{qc2}
h_{G,c}(x,y)\leq \log\left(1+\frac{c t d(z)}{d(z)\sqrt{1-t}}\right)=\log\left(1+c\frac{t}{\sqrt{1-t}}\right),
\eeq
\beq\label{qc3}
h_{G,c}(x,y)\geq \log\left(1+\frac{c t d(z)}{d(z)\sqrt{1+t}}\right)=\log\left(1+c\frac{t}{\sqrt{1+t}}\right).
\eeq
By \eqref{qc0}, and denoting
$$B(f(x),f(y),f(z))=\frac{m_{fG}(f(x),f(z))\sqrt{1+\frac{|f(y)-f(z)|}{m_{fG}(f(y),f(z))}}}{m_{fG}(f(y),f(z))}\,,$$
we get
\bequu
\frac{|f(x)-f(z)|}{|f(y)-f(z)|}& \leq & \frac{U_{fG}(f(x),f(z))\cdot m_{fG}(f(x),f(z)) \sqrt{1+\frac{|f(y)-f(z)|}{m_{fG}(f(y),f(z))}}}{U_{fG}(f(y),f(z))\cdot m_{fG}(f(y),f(z))}\\
&\leq & \frac{e^{L h_{G,c}(x,y)}-1}{e^{ h_{G,c}(x,y)/L}-1}B(f(x),f(y),f(z))\\
&\leq & \frac{e^{L \log\left(1+c\frac{t}{\sqrt{1-t}}\right)}-1}{e^{ \log\left(1+c\frac{t}{\sqrt{1+t}}\right)/L}-1}B(f(x),f(y),f(z))\to L^2,
\eequu
when $t\to 0\, . $
Hence we have the following bound for the linear dilatation
\[
H(f,z)=\limsup_{|x-z|=|y-z|=r\rightarrow 0^+}{\frac{|f(x)-f(z)|}{|f(y)-f(z)|}}\leq L^2 \,.
\]
\end{proof}

\bigskip
{\bf Acknowledgements.} {The research of the first author was supported by a grant received from the Finnish Academy of Science and Letters and also as a part of EUMLS project with grant agreement PIRSES $-$ GA $-$ 2011 $-$ 295164.} The second author was supported by UTUGS, The Graduate School of the University of Turku.
 The authors are indebted to the referee for valuable comments.


\small

\begin{thebibliography}{VVVV}
\bibitem[A]{a}
{\sc L. V.
 Ahlfors}: M\"obius transformations in several dimensions. Ordway Professorship Lectures in Mathematics. University of Minnesota, School of Mathematics, Minneapolis, Minn., 1981. ii+150 pp.

\bibitem[AVV]{avv}
{\sc  G. D. Anderson, M. K. Vamanamurthy, and M. Vuorinen}: Conformal
invariants, inequalities and quasiconformal maps. J. Wiley, 1997.

\bibitem[B]{b} {\sc A. F. Beardon}: The geometry of discrete groups. Graduate texts in
Math., Vol. 91, Springer-Verlag, New York, 1983.
%
\bibitem[CHKV]{chkv}{\sc J. Chen, P. Hariri, R. Kl\'en, and M. Vuorinen}:
Lipschitz conditions, triangular ratio metric, and quasiconformal maps,
{Ann. Acad. Sci. Fenn. 40, 2015, 683--709,} doi:10.5186/aasfm.2015.4039,
{arXiv:1403.6582} [math.CA]

\bibitem[DD]{dd}{\sc M.M.
Deza and E. Deza}: Encyclopedia of distances. Third edition. Springer, Heidelberg, 2014.
xx+733 pp. ISBN: 978-3-662-44341-5; 978-3-662-44342-2




\bibitem [GH]{gh} {\sc F. W. Gehring and  K. Hag}: The ubiquitous quasidisk. With contributions
by Ole Jacob Broch. Mathematical Surveys and Monographs, 184. American Mathematical Society,
Providence, RI, 2012. xii+171 pp.


\bibitem[GO]{go}{\sc  F. W. Gehring and B. G. Osgood}: Uniform domains and the
quasihyperbolic metric. J. Analyse Math. 36 (1979), 50--74 (1980).


\bibitem[GP]{gp}{\sc  F. W.
  Gehring and  B. P. Palka}: Quasiconformally homogeneous domains. J. Analyse Math. 30 (1976), 172--199.


\bibitem[HVW]{hvw}{\sc P. Hariri, M. Vuorinen, and G. Wang}: Some remarks on the visual angle metric. Comput. Methods Funct. Theory, doi: 10.1007/s40315-015-0137-8,
{arXiv: 1410.5943} [math.MG] 12pp.
\bibitem[H]{h}
{\sc  P. H\"{a}st\"{o}}: A new weighted metric, the relative metric I. J. Math. Anal. Appl. 274, (2002), 38--58.




\bibitem[HIMPS]{himps}
{\sc  P. H\"ast\"o, Z. Ibragimov, D. Minda, S. Ponnusamy, and S. K. Sahoo}:
{ Isometries of some hyperbolic-type path metrics, and the hyperbolic medial axis}.
 In the tradition of Ahlfors-Bers, IV, Contemp. Math., 432, Amer. Math. Soc. (2007), 63--74.
%

\bibitem[KLVW]{klvw}
{\sc  R. Kl\'en,  H. Lind\'en, M. Vuorinen, and G. Wang}: {The
visual angle metric and M\"obius transformations}.  Comput. Methods
Funct. Theory 14 (2014), 577--608,  doi 10.1007/s40315-014-0075-x, {arXiv:1208.2871 [math.MG]}.


\bibitem[KRT]{krt}
{\sc R. Kl\'en, A. Rasila, and J. Talponen}: On smoothness of quasihyperbolic balls,  	 arXiv:1407.2403 [math.FA].



\bibitem[MV]{mv}

{\sc O.
 Martio and J. V\"ais\"al\"a}:  Quasihyperbolic geodesics in convex domains II. Pure Appl. Math. Q. 7 (2011), no. 2, Special Issue: In honor of Frederick W. Gehring, Part 2, 395--409.

\bibitem[NA]{na}
{\sc
N. Nikolov and L. Andreev}:
Estimates of the Kobayashi and quasi-hyperbolic distances, arXiv:1510.01571


\bibitem[RT]{rt}
{\sc A.  Rasila and J. Talponen}: On quasihyperbolic geodesics in Banach spaces.
Ann. Acad. Sci. Fenn. Math. 39 (2014), no. 1, 163--173.

\bibitem[V]{v1} {\sc J. V\"ais\"al\"a}:
Lectures on n-dimensional quasiconformal mappings.
Lecture Notes in Math. 229, Springer-Verlag, Berlin, 1971.


\bibitem[Vu1]{vu0} {\sc  M. Vuorinen}:  Conformal invariants and quasiregular mappings. - J. Anal. Math. 45 (1985), 69-115.

\bibitem[Vu2]{vu}
{\sc  M. Vuorinen}: Conformal geometry and quasiregular mappings. Lecture
Notes in Math. 1319, Springer-Verlag, Berlin, 1988.
\bibitem[Vu3]{vu1}
{\sc  M. Vuorinen}: Geometry of Metrics.
Proc. ICM2010 Satellite Conf. International Workshop on Harmonic
and Quasiconformal Mappings (HMQ2010), eds. D. Minda, S. Ponnusamy, N.
Shanmugalingam, J. Analysis 18 (2010), 399--424, ISSN 0971-3611, arXiv:1101.4293 [math.CV].

\end{thebibliography}
\end{document}